\documentclass[11pt,english]{article}
\usepackage[T1]{fontenc}
\usepackage[utf8]{luainputenc}
\usepackage{xcolor}
\usepackage{pdfcolmk}
\usepackage{amsthm}
\usepackage{amsmath}
\usepackage{stmaryrd}
\usepackage{stackrel}
\usepackage{setspace}
\PassOptionsToPackage{normalem}{ulem}
\usepackage{ulem}

\makeatletter

\providecolor{lyxadded}{rgb}{0,0,1}
\providecolor{lyxdeleted}{rgb}{1,0,0}

\numberwithin{figure}{section}
\numberwithin{equation}{section}
\numberwithin{table}{section}

  \input amssymb.sty
\usepackage{etoolbox}
\patchcmd{\thebibliography}{\section*}{\section}{}{}

\usepackage{algpseudocode} \usepackage{algorithm}
\usepackage{algpseudocode,algorithm}
\usepackage{lipsum}
\usepackage{caption}
\usepackage{graphics}

\usepackage{amsmath}
\usepackage{amssymb}
\usepackage[all]{xy}

\usepackage{epsfig}\usepackage{youngtab}

\newcommand{\ef}{\end{equation}}
\chardef\bslash=`\\ 

\newdir{:=}{{}}
\newcommand*\colvec[3][]{
    \begin{pmatrix}\ifx\relax#1\relax\else#1\\\fi#2\\#3\end{pmatrix}
}
\newtheorem{thm}{Theorem}[section]
\newtheorem*{thm*}{Theorem}

\newtheorem{lem}{Lemma}[section]
\newtheorem*{lem*}{Lemma}

\newtheorem*{corl*}{Corollary}

\newtheorem{prop*}{Proposition}

\theoremstyle{definition}
\newtheorem{defn}{Definition}[section]

\newtheorem*{examp*}{Example}
\newtheorem*{remark*}{Remark}
\newtheorem*{CC*}{Crossover Conjecture}
\newtheorem*{Note*}{Note}
\newtheorem*{defn*}{Definition}

 \theoremstyle{remark}
\newtheorem{remark}{Remark}[section]

 \renewcommand{\sectionmark}[1]{}

\makeatother

\usepackage{babel}
\begin{document}

\title{Efficiency of mutation reduction for Brauer trees} 
\author {Zehavit Zvi}

\maketitle

 \begin{abstract}

Brauer tree algebras are important and fundamental blocks
in the modular representation theory of groups.  Aihara develped an algorithm, which we call a mutation reduction,  for getting from a Brauer tree algebra to the simpler Brauer star algebra using a sequence  of mutations centered on edges.  Schaps and Zvi, using the  
 Schaps-Zakay theory of pointing the tree, showed that different algorithms for the sequence of mutations give permutations of the edges. 
 
We give a modification of Aihara's algorithm and test it against the original algorithm for computational efficiency. We prove that all versions of Aihara's algorithm are the fastest possible in the sense of requiring the least number steps to reach the Brauer star, when compared to all possible complete mutation reduction algorithms.

\end{abstract}

\section{INTRODUCTION}

\noindent This work concerns Brauer tree algebras, a widely studied class of algebras of finite representation type which includes all  blocks of cyclic defect 
group in modular group representation
theory. A block of cyclic defect group is a Brauer 
tree algebra and its Green correspondent is a Brauer star algebra.
Rickard proved \cite{R2} that every Brauer tree algebra has a tilting complex which 
makes it derived equivalent to the corresponding Brauer star algebra.  
Schaps-Zakay \cite{SZ1} \cite{SZ2} showed that the tilting complexes in the 
opposite direction can be constructed from irreducible 
projective complexes of length two.  This is the 
all-at-once approach to the theory.

~

\noindent The other main approach is  the step-by-step approach
going back to König and Zimmermann {[}KZ1{]}, later formulated in terms
of mutations by Aihara \cite{Ai}, used by Chan \cite{Ch}
 and Zvonarevna \cite{Zv}, and extended by Kozakai \cite{Ko} to pointed trees.

\section{DEFINITIONS AND NOTATION}

\subsection{Brauer trees}

We first define  the Brauer
tree algebras.

\noindent \begin{defn}

\noindent Let $e$ and $m$ be natural numbers. A \textit{Brauer tree}
of type\textit{ }$(e,m)$\textit{ }is a finite tree $(V,\ensuremath{{\cal E}})$
where $V$ is the set of vertices, $\ensuremath{{\cal E}}$ is the
set of edges, $|\ensuremath{{\cal E}|}=e$, together
with a cyclic ordering of the edges at each vertex and a designation
of an exceptional vertex which is assigned multiplicity $m$. A counterclockwise circuit of the tree is called a \textit{Green's walk}.

\noindent \end{defn}

\noindent The set of all edges incident to vertex $u$ is denoted by $\ensuremath{{\cal E}}(u)$.
By \textquotedbl{}cyclic ordering\textquotedbl{} we mean that for
each edge \textit{E} in $\ensuremath{{\cal E}}(u)$ there is a `next'
edge in $\ensuremath{{\cal E}}(u)$ and that edge has a next edge
in $\ensuremath{{\cal E}(u)}$ etc., until each edge of $u$ is counted
exactly once, in which case\textit{ E} is the next one. We note that
if \textit{E} and $F$ are the only edges of $u$ then $F$ is next
after \textit{E} and \textit{E} is next after $F$.

\begin{defn} \label{primary} The first edge in the cyclic order after the edge which enters the vertex from the path leading from the exceptional vertex will be called the \emph{primary edge}.  The last edge in cyclic ordering before that entering edge is called the \textit{coprimary edge}.
\end{defn}
\noindent \smallskip{}

\noindent Every Brauer tree can be embedded in the plane in such a
way that the cyclic ordering on each $\ensuremath{{\cal E}}(u)$ is
the counterclockwise direction. Two important examples of Brauer trees are:

(i) The \textit{star} with the exceptional vertex in the center.

(ii) The \textit{linear tree}, which includes, for example, the Brauer trees of blocks of cyclic defect in the symmetric groups.

\noindent \smallskip{}

\noindent We relate Brauer trees to the structure of algebras.

\noindent \begin{defn}An algebra $A$ is called a \textit{Brauer
tree algebra} if there is a Brauer tree such that the indecomposable
projective $A$-modules can be described by the following algorithm:\renewcommand{\labelenumi}{(\roman{enumi})}
\begin{enumerate}
\item There is bijection between the edges of the tree and the isomorphism
classes of simple $A$-modules, i.e. each edge is labelled by the
corresponding isomorphism class. 
\item If $S$ is a simple $A$-module and $P_{S}$ is the corresponding
indecomposable projective $A$-module then $P_{S}\supseteq\text{rad}(P_{S})\supseteq\text{soc}(P_{S})\cong S$
and $\text{rad}(P_{S})/\text{soc}(P_{S})$ is a direct sum of one
or two uniserial modules corresponding to the two vertices of the edge, with composition factors determined 
by a clockwise circuit around the vertex.  For edges at the exceptional vertex, 
the clockwise circuit is made $m$ times.
\end{enumerate}
\noindent \end{defn}

\noindent  Even when the algebra which interests us is the block
of a group algebra, we will not use the actual block but rather its skeleton. In this paper we will not need with the algebra structure nor the exact tilting complex giving the mutations. We will be working solely with the combinatorics, whose relations to the Brauer tree algebras and to tilting complexes have been proven in earlier works \cite{SZ,Ko}.

\subsection{Mutation reduction}

Assume we are given a Brauer tree $G$, with multiplicity $m$.  If $m>1$, 
then there is a designated exceptional vertex $v$. For $m=1$, we assume that one of
the vertices has been chosen as the exceptional vertex $v$. Since our graph is a tree,
there is a well-defined distance of each vertex $u$ from $v$ given by 
counting the number of edges on the unique path connecting the vertex to the exceptional vertex. If the edges of the tree are labelled, then each vertex can be given
 the same label as the first edge on this unique path. The distance of an edge is the distance of the vertex farthest from the exceptional vertex.
  
\begin{defn}
The \textit{total distance} $d(G)$ in the tree will be the sum of the distances for all the non-exceptional vertices or the sum of the ditances of all edges. 
\end{defn}
\begin{remark}
For the Brauer star the distance will equal $e$, the number of edges, since each vertex will be at distance $1$. For all other Brauer trees $G$ we have $d(G)>e$.
\end{remark}

\begin{defn} We will describe two dual forms of mutation centered on an edge $i$ of a Brauer tree which is not directly connected to the exceptional vertex. We number the edges at the vertex $u$ of $i$ closest to the exceptional vertex by $k_1,k_2\dots,k_a$ and the edges at the other vertex $v$ by $j_1, j_2, \dots, j_b$.
	\begin{itemize}
		\item  The mutation $\mu^-$. If $k_1$ is the edge immediately before $i$ at the vertex $u$  closest to the exceptional vertex and $j_1$ is the edge immediately preceding $i$ at the far vertex $v$, then we detach $i$ from both vertices and reattach $i$ to the vertex of $k_1$ farther from $i$ in such a way that $k_1$ is now before $i$ in the clockwise direction, and we reattach $i$ to the vertex of $j_1$ farther from $i$ in such a way that $j_1$ is now before $i$ in the clock-wise direction.
		
		\item The mutation $\mu^+$ \cite{Ch}.  This is dual, so that we take the edges $k_a$  after $i$ in the clockwise direction and $j_b$ which is after $i$ in the clockwise direction.  We detach $i$ from both vertices and reattach at the far end of $k_a$ and $j_b$, in such a way that $k_a$ is now before $i$ and $j_a$ is after $i$ in the clockwise direction.
	\end{itemize}
\end{defn}

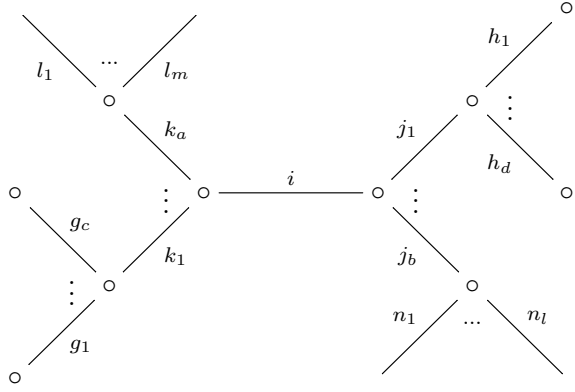
\[
\xymatrix{
	&	&	&	&	&	& {\circ} \\
	& {\circ}   \ar@{-}[ul]^{l_1} \ar@{-}[ur]_{l_m} \ar@{}[u]|(0.4){\dots}	&	&	&	& {\circ} \ar@{-}[ur]^{h_1} \ar@{-}[dr]_{h_d} \ar@{}[r]|(0.4){\vdots}	&	\\
{\circ}	&	& {\circ} \ar@{-}[rr]^i \ar@{-}[dl]^{k_1} \ar@{-}[ul]_{k_a} \ar@{}[l]|(0.4){\vdots}		&	& {\circ} \ar@{-}[ur]^{j_1} \ar@{-}[dr]_{j_b} \ar@{}[r]|(0.4){\vdots} &	& {\circ}  \\
	& {\circ} \ar@{-}[dl]^{g_1} \ar@{-}[ul]_{g_c} \ar@{}[l]|(0.4){\vdots}	&	&	&	& {\circ}  \ar@{-}[ld]_{n_1}  \ar@{-}[rd]^{n_l}  \ar@{}[d]|(0.4){\dots} \\
{\circ}	& & & & & & & & & & & & & &\\ }
\] 
$\mu-:$
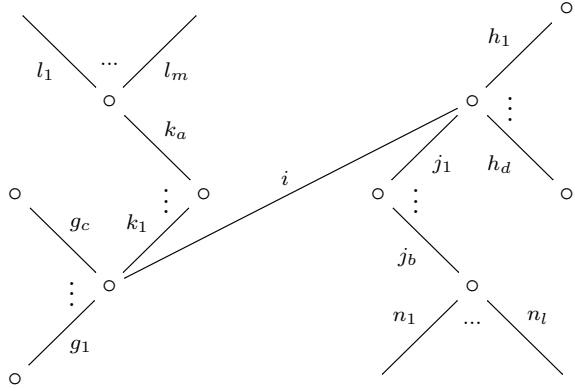
\[
\xymatrix{
	&	&	&	&	&	& {\circ} \\
	& {\circ}	  \ar@{-}[ul]^{l_1} \ar@{-}[ur]_{l_m} \ar@{}[u]|(0.4){\dots}	 &	&	&	& {\circ} \ar@{-}[ur]^{h_1} \ar@{-}[dr]_{h_d} \ar@{}[r]|(0.4){\vdots}	&	\\{\circ}	 &	& {\circ} \ar@{-}[dl]_{k_1}   \ar@{-}[ul]_{k_a} \ar@{}[l]|(0.4){\vdots}	 & & {\circ} \ar@{-}[ur]_{j_1} \ar@{-}[dr]_{j_b}  \ar@{}[r]|(0.4){\vdots}	&	& {\circ} \\
	& {\circ} \ar@{-}[rrrruu]^i  \ar@{-}[dl]^{g_1} \ar@{-}[ul]_{g_c} \ar@{}[l]|(0.4){\vdots}	&	&	&	& {\circ}    \ar@{-}[ld]_{n_1}  \ar@{-}[rd]^{n_l}  \ar@{}[d]|(0.4){\dots} \\
{\circ}	& & & & & & & & & & & & & &
}
\]
$\mu+:$
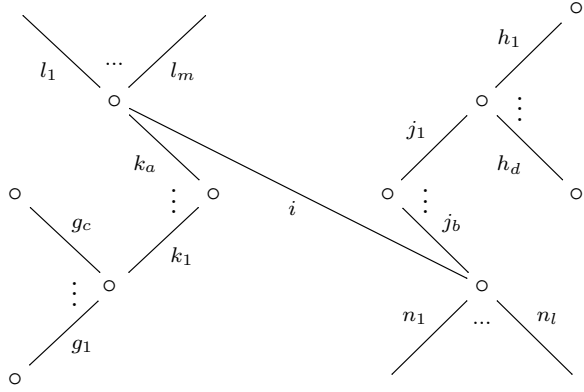
\[
\xymatrix{
	&	&	&	&	&	& {\circ} \\
	& {\circ}   \ar@{-}[ul]^{l_1} \ar@{-}[ur]_{l_m} \ar@{}[u]|(0.4){\dots}	&	&	&	& {\circ} \ar@{-}[ur]^{h_1} \ar@{-}[dr]_{h_d} \ar@{}[r]|(0.4){\vdots}	&	\\{\circ}	 &	& {\circ} \ar@{-}[dl]^{k_1}   \ar@{-}[ul]^{k_a} \ar@{}[l]|(0.4){\vdots}		&	& {\circ}    \ar@{-}[ur]^{j_1} \ar@{-}[dr]^{j_b}  \ar@{}[r]|(0.4){\vdots}	&	& {\circ}	\\
	& {\circ} \  \ar@{-}[dl]^{g_1} \ar@{-}[ul]_{g_c} \ar@{}[l]|(0.4){\vdots}	&	&	&	& {\circ}  \ar@{-}[lllluu]^i    \ar@{-}[ld]_{n_1}  \ar@{-}[rd]^{n_l}  \ar@{}[d]|(0.4){\dots}    \\
{\circ}	& & & & & & & & & & & & & &
}
\]
\noindent \begin{defn}
A \textit{mutation reduction} is a mutation or  sequence of mutations such that the distance of 
each vertex from the exceptional vertex does not ever increase,
and such that the total distance actually decreases. A mutation
reduction which ends at the Brauer star is called \textit{complete}.
\end{defn}

\section{MUTATION REDUCTION ALGORITHMS}

We first quote of a lemma from \cite{SZ}

\begin{lem}\label{branch} \cite{SZ} Assume we are given a Brauer tree.
 
\begin{enumerate}
\item A mutation $\mu^-$ which is a mutation reduction must be centered at a primary edge.
\item A mutation centered at a primary edge connected to an edge adjacent to the 
exceptional vertex is a mutation reduction.

\end{enumerate}
\end{lem}

We now prove a version of this lemma for $\mu^+$. 
\begin{lem}\label{branch+}  Assume we are given a Brauer tree.
	
	\begin{enumerate}
		\item A mutation $\mu^+$ which is a mutation reduction must be centered at a co-primary edge.
		\item A mutation centered at a co-primary edge connected to the exceptional vertex  is a mutation reduction.		
	\end{enumerate}
\end{lem}
\begin{proof}
	\begin{enumerate}
\item If the mutation is not centered at a coprimary edge, then the mutation reattaches
the center at the far end of the edge before it in the cyclic ordering, which is at greater
distance from the exceptional vertex, in contradiction to our assumption that we have a mutation reduction. The only mutation which can reattach at a shorter distance from the exceptional vertex is when it is being reattached to the other end of the entering vertex, which means that the center was a coprimary edge. 
\item
Suppose that  $w$ is a coprimary edge connected to an edge $u$ adjacent to the exceptional vertex, and let $u_0$ be the common vertex of $u$ and $w$. 
Let $U$ be the remaining branches at $u_0$. The effect of mutation  $\mu^{+}$ is to create a new branch by lopping off $w$ and the subgraph $S$ of all edges connected to the exceptional vertex through the center $w$ of the mutation. 

Let $t$ be the coprimary edge at the far end of $w$, and let $T$ be the set of branches at the far end of t. Let $W$
be the remaining branches connected to the far end of $w$. After the mutation of type $\mu^+$, the
original branch rooted at $u$ will now be replaced by two branches, one rooted at $w$ and now
connected directly to $T$, while $t$, which has become coprimary, now has $W$ at its far end. The
other branch will be rooted at $u$, now connected only to $U$ at its far end, and will follow the
branch rooted at $w$ immediately in the counterclockwise ordering at the exceptional vertex  $v$.
It remains to show that this operation was actually a mutation reduction. The edge $w$,
once at distance $2$, is now at distance $1$. The edge $t$, once at distance $3$, is now at distance $2$,
and every edge of $T$, originally connected to $t$ and thus connected to the exceptional vertex
via three edges, $u$, $w$, $t$, is now connected to $v$ via $w$ and therefore every vertex is at distance
two less than before. Finally, the vertices of $W$, now all connected to the exceptional vertex
via $w$, $t$ instead of $u$, $w$, all remain at exactly the same distance that they had before, as will
all the vertices in $U$. The distances of all other edges in the tree are unchanged by the mutation.
\end{enumerate}
\end{proof}

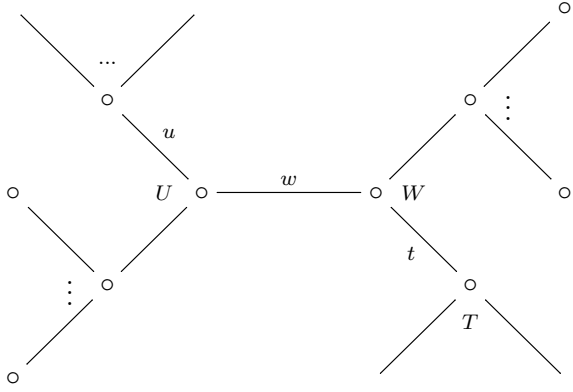
\[
\xymatrix{
	&	&	&	&	&	& {\circ} \\
	& {\circ}   \ar@{-}[ul] \ar@{-}[ur] \ar@{}[u]|(0.4){\dots}	&	&	&	& {\circ} \ar@{-}[ur] \ar@{-}[dr] \ar@{}[r]|(0.4){\vdots}	&	\\
	{\circ}	&	& {\circ} \ar@{-}[rr]^w \ar@{-}[dl] \ar@{-}[ul]_{u} \ar@{}[l]|(0.4){U}		&	& {\circ} \ar@{-}[ur] \ar@{-}[dr]_{t} \ar@{}[r]|(0.4){W} &	& {\circ}  \\
	& {\circ} \ar@{-}[dl] \ar@{-}[ul] \ar@{}[l]|(0.4){\vdots}	&	&	&	& {\circ}  \ar@{-}[ld]  \ar@{-}[rd] \ar@{}[d]|(0.4){T} \\
	{\circ}	& & & & & & & & & & & & & &\\ }
\]

The following lemma was proven in \cite{SZ}, Lemma 3.1(3) for $\mu^-$ alone. We now prove it for a mutation reduction which includes also $\mu^+$
\begin{lem}
	In any complete mutation  reduction, each branch is eventually split entirely into separate
	leaves attached to the exceptional vertex.  All the edges in the original branch correspond to
	an interval around the star and the intervals follow the counterclockwise ordering of the branches.
\end{lem}
	\begin{proof}
		Let us choose a branch and label all the edges of that branch by $\{S_1,\dots, S_t\}$.  A new branch is created if and only if the center is a primary edge attached to a vertex of distance $1$ on which we perform $\mu^-$ or the center is a coprimary edge attached to a vertex of distance $1$ on which we perform $\mu^+$ and the new branch is adjacent to the previous one. Both the old branch and the new contain only edges from the set  $\{S_1,\dots, S_t\}$.  Thus when we finish the complete mutation reduction, the edges of the original branch form an interval, without intervening edges from any other branch. The action by mutation on differenct branches is independent, so we can perform all the mutations on a single branch and then proceed to the next branch in counterclockwise order, reducing it to an interval as well.		
	\end{proof}

\noindent We now give the original algorithm given by Aihara \cite{Ai}.

\noindent \uline{Aihara's Algorithm}\cite{Ai} 
\begin{enumerate}
\item Choose an initial branch.  

\item In a Green's walk starting at the root of the  the initial branch choose the first   primary edge
$w$ attached to  an edge adjacent to the exceptional vertex.  If the tree is not a star, there must
be such an edge. 

\item By Lemma \ref{branch}(2), the mutation centered on this edge $w$ is a mutation
reduction, and from the proof we see that it creates two adjacent branches from the original,
the first of which in counter-clockwise order is rooted at $w$.

\item If $w$ was on the initial branch, let the new initial branch be the new branch rooted at $w$, and otherwise
let the initial branch remain as before.  Begin again from 2.

\end{enumerate}

In \cite{Ko}, Kozakai introduced an algorithm for mutation reduction on pointed Brauer trees in which he uses either $\mu^-$ or $\mu^+$ depending on the pointing.  What we propose is a modification of Aihara's algorithm which uses either $\mu^-$ or $\mu^+$, depending on properties of the primary or coprimary edge connected to an edge adjacent to the exceptional vertex. 

\begin{defn}
	A \textit{generalized Aihara algorithm} is a complete mutation reduction in which every center is a primary edge on which we perform $\mu^-$  or a coprimary edge on which we perform $\mu^-$, where eadh center is attached to an edge adjacent to the exceptional vertex.
\end{defn}

Our original hope was that by careful choice of the edge, we could reduce the total distance more quickly.  To test this possibility,  we created many random graphs and proposed a generalized Aihara algorithm that would use on $\mu{-}$ and $\mu{+}$ depending on the total distance of all edges on each of the two possible branches, the  primary and coprimary. We checked the efficiency of this proposed algorithm against the original Aihara algorithm with $\mu{-}$ on each of the generated trees and discovered that the original algorithm was faster.  Then we asked for the number of steps in the two algorithms and discovered that they were identical.  The new algorithm was only slower because of the time required to check the total distance. This lead to the following theorem.

 \begin{thm}
A generalized Aihara algorithm is the more efficient than any other possible mutation reduction algorithm. Assume we start with a uni-branch Brauer tree.  In terms of number of centers, all generalized Aihara algorithms have the same number, equal to $e-1$, where $e$ is the number of edges in the tree.
 \end{thm}
\begin{proof}
No mutation can add more than than one branch at the exceptional vertex by the defnition of mutation, and the only center which can add such a branch at the exceptional vertex is a mutation reduction by a center attached to an edge adjacent to the exceptional vertex.

Thus the number of new branches at each step is either zero or one and a branch, once created, is not destroyed by any mutation reduction, since the only mutation which removes a branch from the exceptional vertex is a mutation with a center adjacent to the exceptional vertex.

 If the center is always attached to an edge adjacent to the exceptional vertex as in all the generalized Aihara algorithms then at each stage a new branch is created. If one ever takes a center which is not attached to an edge adjacent to the exceptional vertex, then the algoithnm does not create a new branch at each step.  
 
 All the generalized Aihara algorithms create a new branch at each step.  All the possible centers will be exhausted after $e-1$ steps.  No other complete mutation reduction algorithm can match this minimal number of steps and thus the generalized Aihara's  algorithms are more efficient than any others, including Kozakai's \cite{Ko} and Zvi's \cite{Z}.

\end{proof}

\noindent{Department of Computer Science, Shamoon college of Engineering, Be'er Sheva, Israel}

\noindent{Department of Computer Science, Shamoon College of 
Engineering,  Be'er Sheva, 8410802 Israel}.\\
email: zehavzv@sce.ac.il\\

\noindent Acknowledgements: This research was partly funded by a grant from the Research and Development Authority at Shamoon College of Engineering.

\noindent MSC2020: {20C05,20C20, 16E35,18G80}\\
Keywords: {Brauer trees,  tilting complexes, mutations, algorithms}\\

\end{document}